\documentclass[3p,times]{amsart}

\usepackage{amssymb}
\usepackage{amsthm}

\usepackage{enumerate} % for enumerate [(i)]/[(a)]/[(1)]

\usepackage{amsmath}
\usepackage{pdfpages}

\newtheorem{theorem}{Theorem}[section]
\newtheorem{lemma}[theorem]{Lemma}

\newtheorem{main-theorem}[theorem]{Theorem}
\newtheorem*{main-theorem*}{Main Theorem}
\newtheorem*{triangulation-theorem*}{Triangulation Theorem}
\newtheorem{corollary*}{Corollary}

\newtheorem*{problem*}{Problem}

\theoremstyle{definition}

\newtheorem{remark}[theorem]{Remark}
\newtheorem{observation}[theorem]{Observation}

\newtheorem*{question*}{Question}
\newtheorem*{notation*}{Notation}

\newtheorem{definition}[theorem]{Definition}

\newcommand{\rad}{\operatorname{rad}}

\newcommand{\op}{\operatorname{op}}

\newcommand{\La}{\Lambda}
\newcommand{\ba}{\bar{\alpha}}

\usepackage{mathrsfs}

\usepackage{etex} % need because of using both tikz and xymatrix below
\usepackage{dsfont}         % for \mathds{1}
\usepackage[h]{esvect}

\usepackage[matrix,arrow,curve,frame,arc]{xy}

\usepackage{pgf}
\usepackage{tikz}
\usepackage{xcolor}

\newcommand{\vf}{\varphi}
\usetikzlibrary{arrows,shapes,automata,backgrounds,petri,calc}

\newcommand{\tikzAngleOfLine}{\tikz@AngleOfLine}
\def\tikz@AngleOfLine(#1)(#2)#3{%
\pgfmathanglebetweenpoints{%
\pgfpointanchor{#1}{center}}{%
\pgfpointanchor{#2}{center}}
\pgfmathsetmacro{#3}{\pgfmathresult}%
}

\begin{document}

\title{Weighted surface algebras: general version. Corrigendum}

{\def\thefootnote{}
\footnote{The research was supported by the program Research in Pairs by the MFO in 2018, and also by
the Faculty of Mathematics and Computer Science of the Nicolas Copernicus University in Toru\'{n}.}
}

\author[K. Edrmann]{Karin Erdmann}
\address[Karin Erdmann]{Mathematical Institute,
   Oxford University,
   ROQ, Oxford OX2 6GG,
   United Kingdom}
\email{erdmann@maths.ox.ac.uk}

\author[A. Skowro\'nski]{Andrzej Skowro\'nski}
\address[Andrzej Skowro\'nski]{Faculty of Mathematics and Computer Science,
   Nicolaus Copernicus University,
   Chopina~12/18,
   87-100 Toru\'n,
   Poland}
\email{skowron@mat.uni.torun.pl}

\begin{abstract}
This amends the definition of general weighted surface algebras.


\end{abstract}

%\linenumbers

\maketitle

\section{Introduction}
In \cite{WSA-GV} we  generalize the original definition of weighted surface algebras in \cite{WSA} by allowing the possibility that
arrows might not be part of the Gabriel quiver, which  gives a much larger class of algebras. 
This means  that the zero relations need modification, to  make sure that the algebras
are symmetric, and of the appropriate dimension.
We found recently that we had missed one necessary modification for the zero relations. Here 
we give the correct definition, and revise the parts of \cite{WSA-GV}  which 
are affected by this modification.

\section{Weighted surface algebras}\label{sec:context}

Recall that a quiver is a quadruple $Q = (Q_0, Q_1, s, t)$ 
where $Q_0$ is a finite set of vertices, 
$Q_1$ is a finite set of arrows, and 
where $s, t$ are maps $Q_1\to Q_0$ associating 
to each arrow $\alpha\in Q_1$ its source $s(\alpha)$  
and its target $t(\alpha)$. 
We say that
$\alpha$ starts at $s(\alpha)$ and ends at $t(\alpha)$. 
We assume throughout that any quiver is connected. 
Moreover, we fix an algebraically closed field $K$.

Denote by $KQ$ the path algebra of $Q$ over $K$. 
The underlying space has basis the set of all paths in $Q$. 
Let $R_Q$ be the ideal of $KQ$ generated by all paths of length $\geq 1$. 
For each vertex $i$, let $e_i$ be the path of length zero at $i$, then 
the $e_i$ are pairwise orthogonal idempotents, and their sum is the identity of $KQ$. 
We will consider algebras of the form $A=KQ/I$ where $I$ is an ideal of $KQ$ 
which contains $R_Q^m$ for some $m\geq 2$, so that
the algebra is finite-dimensional and basic.
The Gabriel quiver $Q_A$ of $A$ is then the full subquiver of $Q$ 
obtained from $Q$ by removing all arrows $\alpha$ 
%which belong to the ideal $I$.
with $\alpha + I \in R_Q^2 + I$.

A quiver $Q$ is \emph{$2$-regular} if for each vertex $i\in Q_0$ 
there are precisely two arrows starting
at $i$ and two arrows ending at $i$. 
Such a quiver has an involution
on the arrows, $\alpha \mapsto \ba$,  
such that for each arrow $\alpha$, 
the arrow $\ba$ is the arrow $\neq \alpha$ such that $s(\alpha) = s(\ba)$. 
 
  A \emph{biserial  quiver} \  is a pair $(Q, f)$ where $Q$ 
is a  (finite) connected 2-regular quiver,  with at least two vertices, and where 
$f$ is a fixed 
permutation of the arrows such that  $t(\alpha) = s(f(\alpha))$ 
for each arrow $\alpha$. 
The permutation $f$ uniquely determines a permutation $g$ of the arrows, 
defined by $g(\alpha) := \overline{f(\alpha)}$ for any arrow $\alpha$. 
A biserial quiver  $(Q, f)$ 
is a \emph{triangulation quiver}  if $f^3$ is the identity, so that   cycles of $f$ have length $3$ or $1$.

We assume throughout that $(Q, f)$ is a triangulation quiver.
%We fix an algebraically closed field $K$, and we 
We
introduce some notation. 
For each arrow $\alpha$, we fix 
\begin{align*} 
m_{\alpha}\in \mathbb{N}^* &&& \mbox{ a weight, constant on $g$-cycles, and }\cr
c_{\alpha} \in K^*  &&&  \mbox{ a parameter,  constant on $g$-cycles, and define }\cr
n_{\alpha}:=   &&&  \mbox{ the length of the $g$-cycle of $\alpha$, } \cr
B_{\alpha}:=  \alpha g(\alpha)\ldots g^{m_{\alpha}n_{\alpha}-1}(\alpha) &&&   \mbox{ the path  along the $g$-cycle of $\alpha$ 
                     of length $m_{\alpha}n_{\alpha}$}, \cr
 A_{\alpha}:=  \alpha g(\alpha)\ldots g^{m_{\alpha}n_{\alpha}-2}(\alpha) &&&   \mbox{ the path  along the $g$-cycle of $\alpha$ 
                     of length $m_{\alpha}n_{\alpha}-1$.}
\end{align*}

If $p$ is a monomial in $KQ$ we write $|p|$ for the length of $p$. 
For elements $p, q\in \La$ we write $p\equiv q$ if $p=\lambda q$ for some non-zero scalar $\lambda \in K$.

\begin{definition}\label{def:virtual} We say that an arrow $\alpha$ of
	$Q$ is virtual if $m_{\alpha}n_{\alpha}=2$. Note that
	this condition is preserved under the permutation $g$, and that
	virtual arrows form $g$-orbits of sizes 1 or 2. 
\end{definition}

%\medskip

\noindent {\bf Assumption} \label{ass}\normalfont 
For the general %WSA, 
weighted surface algebra
we assume that the following conditions are
satisfied:
\begin{enumerate}[(1)]
 \item
	$m_{\alpha}n_{\alpha}\geq 2$ for all arrows $\alpha$,
 \item
	$m_{\alpha}n_{\alpha}\geq 3$ for all arrows $\alpha$ such
	that $\ba$ is virtual and  $\ba$ is not a loop,
 \item
	$m_{\alpha}n_{\alpha}\geq 4$ for all arrows $\alpha$ 
	such that $\ba$ is virtual and  $\ba$ is a loop.
\end{enumerate}
\normalfont
Condition (1) is a general assumption, and (2) and (3) are needed to eliminate
	two small algebras (see \cite{WSA-GV}).  In particular we exclude the
	possibility that
	both arrows starting at a vertex are virtual, and also
	that both arrows ending at a vertex are virtual.
The  Gabriel quiver $Q_{\La}$ of
$\La$ is obtained from $Q$ by removing
all virtual arrows.

The revised definition of a weighted surface algebra
is now as follows.

\begin{definition} \label{def:2.2}
The algebra $\La = \La(Q, f, m_{\bullet}, c_{\bullet}) = KQ/I$ 
is a weighted surface algebra if
$(Q, f)$ is a triangulation quiver, with $|Q_0| \geq 2$, 
and $I= I(Q, f, m_{\bullet}, c_{\bullet})$ is the ideal of 
$KQ$ generated by:
\begin{enumerate}[\rm(1)]
 \item
	$\alpha f(\alpha) - c_{\ba}A_{\ba}$ for all arrows
	$\alpha$ of $Q$,
 \item
	$\alpha f(\alpha) g(f(\alpha))$ \ for all arrows $\alpha$ of $Q$
	unless  $f^2(\alpha)$ is virtual, or unless $f(\ba)$ is virtual and $m_{\ba}=1, \  n_{\ba}=3$. 
 \item
	$\alpha g(\alpha)f(g(\alpha))$ for all arrows $\alpha$ of $Q$
	unless $f(\alpha)$ is virtual, or unless $f^2(\alpha)$ is virtual and $m_{f(\alpha)}=1, \  n_{f(\alpha)}=3$. 
\end{enumerate}
\end{definition}

\medskip

We assume in this note throughout that $|Q_0|\geq 3$. The details for the only quiver with two vertices are
discussed in 3.1 of \cite{WSA-GV} (and other places), and no correction is needed.
Below we will clarify the exceptions in (2) and (3) of Definition \ref{def:2.2}. 
We observe that with $\La$, also $\La^{\op}$ is a weighted surface algebra, and we will use this to reduce calculations.

\subsection{Some combinatorics related to $g$-cycles of length $2$ or $3$}

We require  the element $A_{\alpha}$ to be non-zero,  and that  $B_{\alpha}$ 
spans the socle of $e_i\La$.  
The problem we had overlooked originally arises when a 3-cycle of $g$ and a 2-cycle with virtual arrows pass through the same vertex.
We will first discuss some combinatorics related to quivers related to 3-cycles and 2-cycles of $g$.

\subsubsection{Virtual arrows}
As discussed in \cite{WSA-GV}, virtual arrows cannot come too close together.
We recall the properties we need and will use frequently:

(1) The arrow  $\alpha$ is virtual 
if and only if $f^2(\ba)$ is virtual. This holds since $f^2(\ba)$ is equal to $g^{-1}(\alpha)$. 

(2) If $\alpha$ is virtual then no other arrow in the $f$-cycle of $\alpha$ is virtual.

(3) If $\alpha$ is virtual then $\ba$ is not virtual.

\bigskip

\subsubsection{Cycles of $g$ of length three}

Recall that $Q$ has at least three vertices. This means that {\it a $g$-cycle of 
length three in $Q$ cannot contain a loop}.  We will used this tacitly in the following. 
We consider $(Q, f)$ where a 3-cycle of $g$ has a common vertex with a 2-cycle of $g$.  This occurs in a subquiver $Q'$ of the form
$$
  
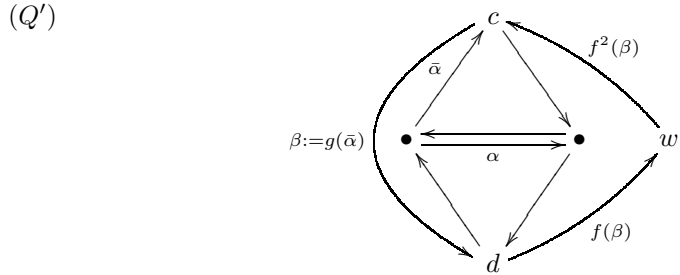
\xymatrix@R=3.pc@C=1.8pc{
    & c 
    \ar[rd]^{}
    \ar@<-.5ex>@/_10ex/[dd]_{\beta:=g(\ba)}
    \\
    \bullet
    \ar[ru]^{\ba}
    \ar@<-.5ex>[rr]_{\alpha}
    && \bullet
    \ar@<-.5ex>[ll]_{}
 \ar[ld]^{}
    & w \ar@/_/[ull]_{f^2(\beta)} 
    \\
    & d
   \ar[lu]^{}
    \ar@/_/[urr]_{f(\beta)}
}\leqno{(Q')}
$$
%{\bf  There has to be an arrow  $\beta:= g(\ba): c \longrightarrow d$ and also arrows
%$f(\beta): d\to w$ and $f^2(\beta): w\to c$ (with a new vertex $w$).}

We have the following observations based on this diagram.

\begin{observation}
\label{obs:2.3}
  \begin{enumerate}[(1)]
	\itemsep0em
	\item 
		The permutation $g$ has  a cycle $(\ba \ \beta \ f^2(\alpha))$.
		Furthermore there are two  $g$-cycles through $c, d$, one of length $3$ and the other of length $d\geq 5$.
	\item 
		It is not possible that $n_{\ba}= n_{f(\ba)}=3$ and $n_{\alpha} (= n_{f^2(\ba)} ) = 2.$
	\item 
		It is not possible that $n_{f(\ba)}=2$ and $n_{\ba}= n_{f(\alpha)}=3$. 
  \end{enumerate}
  \end{observation}

\medskip

\begin{remark}\label{rem:2.4}\normalfont 
A triangulation quiver $(Q, f)$ can have an arbitrary number of subquivers isomorphic to $Q'$. For example, start with a cyclic quiver which has
vertices $1, 2, \ldots, n$. For each $i$, attach a copy of $Q'$ at vertex $i$ by identifying $i$ with $w$. Then, using the $*$-construction introduced
in \cite{ESk-BGA}, one can extend this quiver to a triangulation quiver. (That is, one splits each of the arrows of the cyclic quiver and adds a new arrow, to
produce triangles, see 4.1 in \cite{ESk-BGA}).
\end{remark}

\bigskip

The following discusses relations (2) and (3) near a loop of the quiver.
They are an easy consequence of the conditions, we omit a proof.

\medskip

\begin{lemma}\label{lem:2.5} 
%Assume $\zeta = \alpha f(\alpha) g(f(\alpha))$ 
%and $\xi = \alpha g(\alpha) f(g(\alpha))$.\\
%(i) If $\alpha$ is a loop then $\zeta = 0 = \xi$.\\
%(ii) Suppose $f(\alpha)$ is a loop, then $\zeta = 0$.\\
%(iii) Suppose $g(\alpha)$ is a loop, then $\xi = 0$.
Assume $\zeta = \alpha f(\alpha) g(f(\alpha))$ 
and $\xi = \alpha g(\alpha) f(g(\alpha))$.
\begin{enumerate}[\rm(i)]
 \itemsep0em
 \item 
  If $\alpha$ is a loop then $\zeta = 0 = \xi$.
 \item 
  Suppose $f(\alpha)$ is a loop, then $\zeta = 0$.
 \item 
  Suppose $g(\alpha)$ is a loop, then $\xi = 0$.
\end{enumerate}
\end{lemma}

\section{The exceptions in (2) and (3)}

Consider $\zeta:= \alpha f(\alpha) g(f(\alpha))$, we determine the exceptions occuring 
in Definition \ref{def:2.2}(2).  
We require that
elements of the form $A_{\beta}$ 
are non-zero in $\La$.

\medskip

\begin{lemma}\label{lem:3.1}  
The element $\zeta$ is a non-zero multiple 
of a path $A_{\alpha}$ or $A_{\ba}$ in 
two cases:\\ 
{\rm(a)} \  The arrow   $\ba$ is virtual, then   
	$\zeta = c_{\ba}c_{\alpha}A_{\alpha}$.\\
{\rm(b)} \  We have  $n_{\ba}=3=m_{\ba}n_{\ba}$ and $f(\ba)$ is virtual, then 
	$\zeta = c_{\ba}c_{f(\ba)} c_{\alpha}A_{\alpha}$.
\end{lemma}

%\medskip
%
%{\it Proof}  
\begin{proof}
By Lemma \ref{lem:2.5} we may assume that $\alpha$ and $f(\alpha)$ are not loops. % ({\tt maybe not needed}). 
By  relation (1)  of Definition \ref{def:2.2} we have 
$$\zeta = c_{\ba}A_{\ba}g(f(\alpha)). \leqno{(*)}
$$

(a) Assume $\ba$ is virtual, 
 then $g(f(\alpha)) = f(\ba)$ and  we have 
 $\zeta = c_{\ba}\ba g(f(\alpha)) =  c_{\ba}\ba f(\ba) = c_{\ba}c_{\alpha}A_{\alpha}.
$

\smallskip

(b) 
Suppose  $n_{\ba}= m_{\ba}n_{\ba} = 3$,  then $g(f(\alpha)) = f(g(\ba))$ and 
(*) is equal to 
$$ c_{\ba}\ba g(\ba) f(g(\ba)) = c_{\ba} c_{f(\ba)} \ba A_{f(\ba)}.
\leqno{(**)}$$
If $f(\ba)$ is virtual then (**) is 
%a non-zero scalar multiple of $A_{\alpha}$.
equal to 
$c_{\ba} c_{f(\ba)} \ba f(\ba) = c_{\ba} c_{f(\ba)} c_{\alpha} A_{\alpha}$.

Suppose $|A_{f(\ba)}|=2$, then $(**)$ is equal to
$$c_{\ba}c_{f(\ba)}\ba f(\ba)g(f(\ba)) = \lambda A_{\alpha}g(f(\ba))
$$ for a non-zero scalar $\lambda$. 
By Observation~\ref{obs:2.3}(2),  
the arrow $\alpha$ is not virtual. 
%Furthermore if $n_{\alpha} = 3$ then by Lemma \ref{lem:2.3} the quiver $Q$ is tetrahedral, 
%which has been dealt with in \cite{WSA}.
The cases of $A_{f(\ba)}$ of higher length will be
dealt with in Lemma~\ref{lem:4.5}.
%$\Box$
\end{proof}

%\bigskip
\medskip

Consider 
$\xi := \alpha g(\alpha)f(g(\alpha))$, 
we determine the exceptions occuring in Definition \ref{def:2.2} (3). We include the proof although it 
is equivalent  to Lemma \ref{lem:3.1} for the opposite algebra of $\La$.

%\medskip

\begin{lemma}\label{lem:3.2} The element $\xi$ is a non-zero scalar multiple of a path $A_{\alpha}$ or
$A_{\ba}$ in the following cases:\\
{\rm(a)}  \ The arrow  $f(\alpha)$ is virtual, and then
	$\xi = c_{f(\alpha)}c_{\ba}A_{\ba}$.\\
{\rm(b)} \ We have $n_{f(\alpha)}m_{f(\alpha)} = 3 = n_{f(\alpha)}$ and $\ba$ is virtual, then $\xi = c_{f(\alpha)} c_{\ba} c_{\alpha} A_{\alpha}$.
\end{lemma}

%\medskip

%{\it Proof} \  
\begin{proof}
We may assume that $\alpha, g(\alpha)$ are not loops, %{\tt maybe not needed}
by Lemma~\ref{lem:2.5}.
By relation (1) of Definition \ref{def:2.2} we have
$$\xi = c_{f(\alpha)}\alpha  A_{f(\alpha)}.
\leqno{(*)}
$$

(a) Assume $f(\alpha)$ is virtual, then 
$\xi = c_{f(\alpha}\alpha f(\alpha) = c_{f(\alpha)}c_{\ba}A_{\ba}$.

(b) Assume now that $n_{f(\alpha)}=n_{f(\alpha)}m_{f(\alpha)} = 3$, 
then (*) is equal to
 $c_{f(\alpha)}\alpha f(\alpha)g(f(\alpha))$, which is 
$$c_{f(\alpha)}c_{\ba}A_{\ba}g(f(\alpha)).
\leqno{(**)}
$$
 If $\ba$ is virtual then $g(f(\alpha)) = f(\ba)$ and  (**)
is 
%a non-zero scalar multiple of $A_{\alpha}$.
equal to $c_{f(\alpha)}c_{\ba}c_{\alpha}A_{\alpha}$.
Suppose $|A_{\ba}|=2$, then (**) is 
$\equiv \ba g(\ba) f(g(\ba)) \equiv \ba A_{f(\ba)}$. 
We have here $|A_{\ba}|=2=|A_{f(\alpha)}|$ and therefore by Observation~\ref{obs:2.3}(3), we know 
$f(\alpha)$ is not virtual.
The cases of $A_{\ba}$ of higher length will be
dealt with in Lemma~\ref{lem:4.5}.
%$\Box$
\end{proof}

%\bigskip
\medskip

Recall that in general $B_{\alpha}\equiv B_{\ba}$ for $\alpha, \ba$ starting at $i$. Namely
$$
 B_{\alpha} = \alpha A_{g(\alpha)}
 \equiv \alpha f(\alpha)f^2(\alpha)
 \equiv A_{\ba} f^2(\alpha)  = B_{\ba}.
$$

\subsection{The socle and the second socle near exceptional $\zeta, \xi$}

Let $A_{\alpha}$, or $A_{\ba}$ be one of the elements occuring in the exceptions described in Lemmas \ref{lem:3.1}
and Lemma \ref{lem:3.2}. 
We will now prove that this belongs to  the second socle. 
This will follow quickly from the following lemma, 
which also will be  useful later.

\bigskip

In the following we assume $\ba$ is virtual;  one may consult one of the following two diagrams, depending on whether or not
$\ba$ is a loop.  
In the first diagram, $g$ has cycles $(f^2(\alpha) \ \ba)$ and $(\beta \ f^2(\ba) \ \alpha \ldots)$,
where $\beta$ is the last arrow in $A_{\alpha}$. 
In the second diagram, $g$ has cycles $(\ba) (f(\alpha) \ \alpha \ g(\alpha) \ldots)$. 
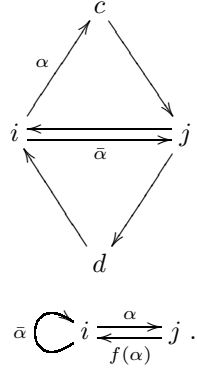
$$
  \xymatrix@R=3.pc@C=1.8pc{
    & c
    \ar[rd]^{}
    \\
    i
    \ar[ru]^{\alpha}
    \ar@<-.5ex>[rr]_{\ba}
    && j
    \ar@<-.5ex>[ll]_{}
 \ar[ld]^{}
    \\
    & d
   \ar[lu]^{}
}
$$
\[
  \xymatrix{
%  \xymatrix@C=1pc{
 i
    \ar@<.5ex>[r]^{\alpha}
    \ar@(dl,ul)^{\bar{\alpha}}[]
& j
    \ar@<.5ex>[l]^{f(\alpha)}
  }
 .
\]

\begin{lemma} \label{lem:3.3} Assume $\ba$ is a virtual arrow. If $\ba$ is not a loop then 
there are six  relations of type $\zeta$ or $\xi$  in which 
$\ba$ occurs. 
If $\ba$ is a loop then there are four relations of type $\zeta$ or $\xi$ in which $\ba$ occurs.
In both cases, each of these is zero in $\La$. 
\end{lemma}

%\medskip
%
%{\it Proof } 
\begin{proof}
We write down details for the case $\ba$ is not a loop, the other case is easier.\\
(1) We start with the three elements of type $\zeta$.

(a) We have $\ba f(\ba) g(f(\ba))=0$ since 
$\alpha$ is not virtual and $f(\alpha)$ is not virtual.

(b) We have $f^2(\ba)\ba f^2(\alpha)=0$. Namely $\overline{f^2(\ba))} = g(f(\ba))$ and this is not virtual since
$f(\ba)$ is not virtual. Suppose $n_{g(f(\ba))} = n_{g(f(\ba))}m_{g(f(\ba))}=3$, then $g$ has cycle
$$(f(\ba) \ g(f(\ba)) \ f(\alpha)).\leqno{(*)}$$
Let $\gamma= f(g(f(\ba))$, we want that $\gamma$ is not virtual. Suppose for a contradiction it is virtual
then also $f^2(\bar{\gamma})$ is virtual. Now $\bar{\gamma}= g^2(f(\ba))= f(\alpha)$
by (*), and $f^2(\bar{\gamma})=\alpha$, which is not virtual, a contradiction.

(c) We have $f(\alpha)f^2(\alpha)\ba = 0$ since $g(\alpha)=\overline{f(\alpha)}$ is not virtual, and if 
$n_{g(\alpha)}=n_{g(\alpha)}m_{g(\alpha)} =3$ then $g$ has cycle 
$$(\alpha \ g(\alpha) \ f^2(\ba)). \leqno{(**)}
$$
If $f(g(\alpha))$ would be virtual then also 
$f^2(\overline{f(g(\alpha))})$ would be virtual.
But using (**) we have
$$f^2(\overline{f(g(\alpha))}) = f^2(g^2(\alpha)) = f^4(\ba) = f(\ba) 
$$
which is not virtual.
\medskip

\noindent
(2) We consider the three elements of type $\xi$. 

(a) We have $f^2(\alpha)\ba f(\ba)=0$ since $\alpha = f(f^2(\alpha))$ is not virtual, and also
$f(\alpha)$ is not virtual. 

(b) We have $\ba f^2(\alpha)\alpha=0$ since $f(\ba)$ is not virtual and $\alpha$ is not virtual.

(c) Let $\beta$ be such that $g(\beta)= f^2(\ba)$, we claim that
$\beta f^2(\ba)\ba = 0$. We have  $f(\beta) = g(f(\ba))$. If this were virtual then also
$f^2(\overline{f(\beta)}) = f^2(g(\beta)) = f^4(\ba)= f(\ba)$ would be virtual which is not the case.
Suppose $n_{f(\beta)} = n_{f(\beta)}m_{f(\beta)} =3$, then $g$ has a cycle 
$$(f(\alpha) \ f(\ba) \ f(\beta)).
$$
Let $s=s(\beta)$, then $f^2(\beta): c\to s$. Therefore $f^2(\beta) = g(\alpha)$
and is not virtual. 
%It follows that $\bar{\beta}$ also is not virtual.
%$\Box$
\end{proof}

%\bigskip
\medskip

\section{Revising \cite{WSA-GV}, Section 4}

Write $\rad \La = J$.

\begin{lemma} \label{lem:4.1}   Let $\alpha$ be an arrow such that $A_{\alpha}\equiv \zeta = \alpha f(\alpha)g(f(\alpha))$ where $\ba$ is virtual, as 
in Lemma \ref{lem:3.1}. Then\\
{\rm(1)} \ $A_{\alpha}J = \langle B_{\alpha}\rangle$ and $JA_{\alpha} = \langle B_{f^2(\ba)}\rangle$.\\
{\rm(2)} \  $B_{\alpha}J=0 = JB_{\alpha}$ and $B_{f^2(\ba)}J=0=JB_{f^2(\ba)}$. 
\end{lemma}

%\medskip
%
%{\it Proof } 
\begin{proof}
We prove these by applying 
Lemma \ref{lem:3.3} repeatedly.

(1) Clearly $A_{\alpha}f^2(\ba) = B_{\alpha}$ and $f^2(\ba)A_{\alpha} = B_{f^2(\ba)}$. 

Next, we have $A_{\alpha}g(f(\ba))  \equiv \ba f(\ba) g(f(\ba)) = 0$ and $f^2(\alpha)A_{\alpha} \equiv f^2(\alpha)\ba f(\ba) = 0$
since they are paths of type $\zeta$, or $\xi$ which involve a virtual arrow. 

(2)  We have $B_{\alpha}\alpha \equiv B_{\ba}\alpha \equiv \ba f^2(\alpha)\alpha = 0$ and similarly
$f^2(\ba) B_{\alpha}=0$. 

Moreover $B_{\alpha}\ba \equiv \alpha[f(\alpha)f^2(\alpha)\ba] =0$ and 
$f^2(\alpha)B_{\alpha} \equiv [f^2(\alpha)\ba f(\ba)]f^2(\ba) =0$. 

(3)   First $B_{f^2(\ba)}f^2(\ba) \equiv f^2(\ba)B_{\alpha} =0$ by (2).
Moreover $B_{f^2(\ba)}g(f(\ba)) = f^2(\ba)A_{\alpha} g(f(\ba)) = 0$ by (1). 
Next, we have using $f^2(\alpha)$ also is virtual, 
$$f(\ba)B_{f^2(\ba)} \equiv f(\ba)f^2(\ba) \ba f(\ba) \equiv f^2(\alpha)\ba f(\ba) = 0.
$$
Finally,
if $\beta = g^{-1}(f^2(\ba))$, then
$\beta B_{f^2(\ba)} \equiv  [\beta f^2(\ba) \ba] f(\ba) = 0$. 
%$\Box$
\end{proof}

%\medskip

\begin{lemma} \label{lem:4.2}   Let $\alpha$ be an arrow such that $A_{\alpha}\equiv \zeta = \alpha f(\alpha)g(f(\alpha))$,
where  $n_{\ba} = n_{\ba}m_{\ba} = 3$ and $f(\ba)$ is virtual, as 
in Lemma \ref{lem:3.1}. Then\\
{\rm(1)} \ $A_{\alpha}J = \langle B_{\alpha}\rangle$ and $JA_{\alpha} = \langle B_{f^2(\ba)}\rangle$.\\
{\rm(2)} \  $B_{\alpha}J=0 = JB_{\alpha}$ and $B_{f^2(\ba)}J=0=JB_{f^2(\ba)}$. 
\end{lemma}

%{\it Proof} 
\begin{proof}
As above, we prove these by applying  Lemma \ref{lem:3.3} repeatedly.

(1) Clearly $A_{\alpha}f^2(\ba) = B_{\alpha}$ and $f^2(\ba)A_{\alpha} = B_{f^2(\ba)}$. 
We have $A_{\alpha}g(f(\ba)) \equiv \ba f(\ba) g(f(\ba)) = 0$ and $f^2(\alpha)A_{\alpha} \equiv f^2(\alpha)\ba f(\ba) = 0$. 

(2) \ We have 
\begin{align*} B_{\alpha}\alpha & \equiv \ba [f(\ba)f^2(\ba)\alpha] = 0,\cr
f^2(\alpha) B_{\alpha} &\equiv [f^2(\alpha)\ba f(\ba)]f^2(\ba) = 0,\cr
B_{\alpha}\ba &\equiv \ba f(\ba)[f^2(\ba)\ba] \equiv \ba f(\ba)g(f(\ba))=0,\cr
f^2(\ba)B_{\alpha} &\equiv [f^2(\ba)\ba] f(\ba)f^2(\ba) \equiv g(f(\ba))f(\ba)f^2(\ba) = 0.
\end{align*} 

(3) \ We have using also some identities from (2) and (1) 
$$B_{f^2(\ba)}f^2(\ba)  = f^2(\ba)B_{\alpha} = 0 \  \mbox{ and } \ 
B_{f^2(\ba)}g(f(\ba)) = f^2(\ba)[A_{\alpha}g(f(\ba))] = 0,$$
\begin{align*} f(\ba)B_{f^2(\ba)} &= [f(\ba)f^2(\ba)\ba]f(\ba) \equiv B_{f(\ba)}f(\ba)\cr
&  \equiv B_{g(\ba)}f(\ba) \equiv g(\ba)[g^2(\ba)\ba f(\ba)] = 0,\cr
\beta B_{f^2(\ba)} & \equiv f(g(\ba)) B_{g(f(\ba))} \equiv f(g(\ba)) g(f(\ba))f(\ba) 
= f(g(\ba)) f^2(g(\ba)) g\big(f^2(g(\ba))\big) = 0.
\end{align*} 
%$\Box$
\end{proof}

The analogues of Lemma \ref{lem:4.1} and Lemma \ref{lem:4.2} for $\xi$ as in Lemma \ref{lem:3.2} hold, they are just the same as  
Lemma \ref{lem:4.1} and Lemma \ref{lem:4.2} for the opposite algebra of $\La$. We give the statements.

\begin{lemma}\label{lem:4.3}  Let $\alpha$ be an arrow such that $A_{\ba}\equiv \xi = \alpha g(\alpha)f(g(\alpha))$,
where  $f(\alpha)$ is virtual, as 
in Lemma \ref{lem:3.2}. Then\\
{\rm(1)} \ $A_{\ba}J = \langle B_{\ba}\rangle$ and $JA_{\ba} = \langle B_{f^2(\alpha)}\rangle$.\\
{\rm(2)} \  $B_{\ba}J=0 = JB_{\ba}$ and $B_{f^2(\alpha)}J=0=JB_{f^2(\alpha)}$. 
\end{lemma}

\medskip

\begin{lemma}\label{lem:4.4}  Let $\alpha$ be an arrow such that $A_{\alpha}\equiv \xi = \alpha g(\alpha)f(g(\alpha))$,
where  $n_{f(\alpha)} = n_{f(\alpha)}m_{f(\alpha)}=3$  and $\ba$ is virtual, as 
in Lemma \ref{lem:3.2}. Then\\
{\rm(1)} \ $A_{\alpha}J = \langle B_{\alpha}\rangle$ and $JA_{\alpha} = \langle B_{f^2(\ba)}\rangle$.\\
{\rm(2)} \  $B_{\alpha}J=0 = JB_{\alpha}$ and $B_{f^2(\ba)}J=0=JB_{f^2(\ba)}$. 
\end{lemma}

We will now review Lemma 4.5 of \cite{WSA-GV}.
It needs a minor modification.

\begin{lemma}
\label{lem:4.5}
Let $\alpha$ be an arrow in $Q$.
Then the following hold:
\begin{enumerate}[\rm(i)]
	\itemsep0em
 \item
  $B_{\alpha} \rad \Lambda = 0$.
 \item
  $B_{\alpha}$ is non-zero.
\item If $\alpha$ is not virtual and we do not have $n_{\alpha}=n_{\alpha}m_{\alpha}=3$ and $f(\alpha)$ virtual then $A_{\alpha}{\rm rad}^2\La = 0$.
\item Suppose $\alpha$ is virtual, or we have both $n_{\alpha}=n_{\alpha}m_{\alpha}=3$ and $f(\alpha)$ being virtual. Then $A_{\alpha}{\rm rad}^2\La = \langle B_{\alpha}\rangle$.
\end{enumerate}
\end{lemma}

%{\it Proof } 
\begin{proof}
 (i) Let $\alpha$ be an arrow. 
 We have proved that $B_{\alpha}J=0$ when 
$\ba$ is virtual and when $f(\alpha)$ is virtual  (in Lemmas  \ref{lem:4.1} and \ref{lem:4.2}). 
Interchanging $\alpha, \ba$ we have  proved $B_{\ba}J=0$  when  $\alpha$ is virtual and when $f(\ba)$ is virtual. 
Now, $B_{\alpha}\equiv B_{\ba}$. Therefore to complete the proof of (i) we may assume 
that none of $\alpha, \ba, f(\alpha)$ and $f(\ba)$ is virtual.

We have 
\begin{align*} B_{\alpha}\alpha  \equiv &\alpha f(\alpha)f^2(\alpha) \alpha\cr
\equiv & \alpha g(\alpha) f(g(\alpha)) f^2(g(\alpha)).
\end{align*}
The product of the first three arrows is zero; it is not one of the exceptions.
Similarly
$$
B_{\alpha}\ba \equiv \ba f(\ba)f^2(\ba)\ba
\equiv \ba g(\ba)f(g(\ba))f^2(g(\ba))
= 0.
$$

(ii) This follows 
from the relations.

(iii) With the assumptions we have $A_{\alpha}g(f(\ba)) \equiv \ba f(\ba)g(f(\ba)) = 0$, it is not an exception, and
$A_{\alpha}f^2(\ba)= B_{\alpha}$. 

(iv) 
Assume first $\alpha$ is virtual, we consider the four generators for $A_{\alpha}{\rm rad}^2\La$.
Two of them are zero by Lemma \ref{lem:3.3}, and $\alpha f(\alpha)f^2(\alpha) \equiv B_{\alpha}$, and 
furthermore $\alpha g(\alpha)g^2(\alpha) \equiv B_{\alpha} g^2(\alpha) =0$ by part (i).

Now assume $n_{\alpha} = n_{\alpha}m_{\alpha} = 3$ and $f(\alpha)$ is virtual. 
Then  
$$A_{\alpha}J^2 = \alpha g(\alpha)g^2(\alpha)J + \alpha g(\alpha)f(g(\alpha))J = 0 + \alpha f(\alpha)J$$
Now $\alpha f(\alpha)J = \langle B_{\alpha}\rangle $ since 
$\alpha f(\alpha) g(f(\alpha)) = 0$ by Lemma \ref{lem:3.3}.
%$\Box$
\end{proof}

\bigskip

%{\tt  I have checked 4.7 which I think is OK. I am not sure about 4.8, I need to look at it again.}

For general weighted surface algebras, for some parameters the socle of $e_i\La$ contains an element which is not a multiple of $B_{\alpha}$. These
have to be identified and excluded. One part of the proof in \cite{WSA-GV} involves relations (2) and (3), we will now revise this.
\medskip

\begin{lemma} \label{lem:4.6}[Lemma 4.10 of \cite{WSA-GV}]
Assume $\alpha, \ba$ are not virtual. Let  $\zeta = \zeta_1+a\zeta_2$  where $\zeta_1, \zeta_2$ are  initial submonomials of $B_{\alpha}, B_{\ba}$ 
and $a\in K^*$ such that $\zeta J=0$ but $\zeta\not\in \langle B_{\alpha}\rangle$. Then 
$\zeta_1=\alpha g(\alpha)$, $\zeta_2 = \ba g(\ba)$ and both $f(\alpha)$ and $f(\ba)$ are virtual. Moreover, for certain parameters, such 
$\zeta$ exists.\end{lemma}

\medskip

%{\it Proof } \   
\begin{proof}
Let $\zeta_1 = \alpha g(\alpha)\ldots g^r(\alpha)$ and $\zeta_2=\ba g(\ba)\ldots g^s(\ba)$. We assume $\zeta J=0$ and $\zeta$ is not a scalar
multiple of $B_{\alpha}$. Then $\zeta_1$ and $\zeta_2$ end at the same vertex, and the arrows starting at this vertex are
$g^{r+1}(\alpha)$ and $g^{s+1}(\ba)$
and as well $f(g^r(\alpha))$ and $f(g^s(\ba))$. It follows that
$g^{r+1}(\alpha) = f(g^s(\ba)), \ \ f(g^r(\alpha)) = g^{s+1}(\ba).
$
This means we have the two identities 
\begin{align*}
\alpha g(\alpha) \ldots  g^r(\alpha)f(g^r(\alpha)) + a \ba g(\ba)\ldots g^{s+1}(\ba)& = 0, \cr
\alpha g(\alpha) \ldots  g^r(\alpha)g^{r+1}(\alpha)  + a \ba g(\ba)\ldots f(g^{s}(\ba)) &= 0 .
%\leqno{(2)}
\end{align*}
Note that in each case, the individual terms must be non-zero (a proper submonomial of $B_{\alpha}$ or $B_{\ba}$ is non-zero).
Now, an initial proper submonomial of $B_{\alpha}$, respectively $B_{\ba}$, 
can only occur in a relation if it is equal to $A_{\alpha}$, respectively $A_{\ba}$. 

{\sc Case 1. } Assume $r, s\geq 1$.
Note that the second monomial in the first equation is $A_{\ba}$. 
We premultiply the first equation with $f^2(\alpha)$,  this gives
$$f^2(\alpha)\alpha g(\alpha) \eta + a B_{f^2(\alpha)} = 0.
$$
Here $\eta = g^2(\alpha)\ldots f(g^r(\alpha))$. 
Since $B_{f^2(\alpha)}\neq 0$ we deduce that $\mu:= f^2(\alpha)\alpha g(\alpha)\neq 0$. So 
this is an exception for the relation (2). By Lemma \ref{lem:3.1}, we have $\mu \equiv A_{f^2(\alpha)}$ and by Lemma \ref{lem:4.1} 
$\mu J^2=0$. But $\mu\eta\neq 0$, in fact is a multiple of $B_{f^2(\alpha)}$, and therefore $r=1$. 
With this, the first equation is
$$\alpha g(\alpha)f(g(\alpha)) + aA_{\ba}=0.\leqno{(*)}$$
By Lemma \ref{lem:3.2} and since $\ba$ is not virtual we have $f(\alpha)$ is virtual and  
$(*)$ holds with $a= -c_{f(\alpha)}c_{\ba}$. 

The same argument for the second equation  shows that $s=1$. Moreover, 
the corresponding identity holds  with $a= -(c_{\alpha}c_{f(\ba)})^{-1}$. 
So when $r=s=1$ and these two parameters are equal we have indeed such an element $\zeta$.

We must show that otherwise no such $\zeta$ exists.

{\sc Case 2. } Assume $r=s=0$. Then the identities are 
$$\alpha f(\alpha) + a \ba g(\ba) =0, \ \ \alpha g(\alpha) + a \ba f(\ba) = 0.
$$
Here $f(\alpha)=g(\ba)$ and $f(\ba) = g(\alpha)$, they are not virtual. Moreover, $|A_{\alpha}|=2$ and $|A_{\ba}|=2$. 
Not both of $\alpha, \ba$ can be loops, say $\alpha$ is not a loop. Then $g$ has a cycle
$(\alpha \ g(\alpha) \ f^2(\ba))$.  Since  $g(\alpha)= f(\ba)$ we have
the contradiction that 
$$g(f(\ba)) = g(g(\alpha)) = f^2(\ba) = f(f(\ba)).$$

{\sc Case 3. } Assume $r  \geq 1$ and $s=0$,
then the argument in Case 1 applies to $r$ and we can deduce that $r=1$. Moreover we have
$|A_{\ba}|=2$ and $f(\alpha)$ is virtual.
We show that this cannot happen. 
The two equations are
\begin{align*} \alpha g(\alpha)f(g(\alpha)) + a \ba g(\ba) &=0,\cr
\alpha g(\alpha) g^2(\alpha) + a \ba f(\ba)&=0.
\end{align*}
In particular $\xi:= \alpha g(\alpha)f(g(\alpha))\neq 0$. We must have case (a) of Lemma \ref{lem:3.2}, so that $f(\alpha)$ is virtual and
$|A_{\ba}|=2$. We also have $|A_{\alpha}|=3$,  but this contradictions Observation~\ref{obs:2.3}(1) which says that $|A_{\alpha}|$ must be at
least $5$. 
This completes the proof.
%$\Box$
\end{proof}

\bigskip

\section{Revising Sections 5 and 6 of \cite{WSA-GV}}

The modified definition plays a role for Lemma 5.5,  and Lemma 5.6 of \cite{WSA-GV} and we will now revise these.

Assume first that $\ba$ is  a virtual loop, then $\alpha$ is not virtual.
Then the
quiver $Q$ contains a subquiver
\[
 \xymatrix{
  i \ar@(dl,ul)[]^{\ba} \ar@<+.5ex>[r]^{\alpha}
   & j \ar@<+.5ex>[l]^{f(\alpha)}  
 }
\]
and $f$ has a cycle $(\ba \ \alpha \ f(\alpha))$. 
Note that by Condition (3) of the Assumption (preceding Definition \ref{def:2.2}) we have $m_{\alpha}n_{\alpha}\geq 4$. 

Let $\gamma$ be the other arrow 
starting at vertex $j$, and $\delta$ be the other arrow ending at $j$. 

\begin{lemma}\label{lem:5.1} Assume $\La$ is not the singular triangle algebra. Then there is an exact sequence of $\La$-modules
	$$0 \to \Omega^{-1}(S_i) \to P_j \to P_j \to  \Omega(S_i)\to 0
	$$
	which gives rise to a periodic minimal projective resolution of $S_i$ 
	in ${\rm mod}\La$. In particular $S_i$ is periodic of period $4$.
\end{lemma}

\begin{proof}
	We have $\Omega(S_i) = \alpha \La$, and we take $\Omega^2(S_i)$ as 
$$
  \Omega^2(S_i) = \{ x\in e_j\La \mid \alpha x = 0\}.
$$
As in the proof of \cite[Lemma~5.5]{WSA-GV} we see that $\vf\La \subseteq \Omega^2(S_i)$ where  
$$\vf:= f(\alpha)\alpha  - c_{\ba}c_{\alpha} A_{\alpha}'$$
with  $\alpha A_{\alpha}' = A_{\alpha}$. 
The module $\Omega^2(S_i)$ has dimension $m_{\alpha}n_{\alpha}-1$. 
We will  show that 
$\vf \La$ has the same dimension,  which will give equality.

\smallskip

First we observe that $\vf f(\alpha)=0$. Namely
$$\vf f(\alpha) = f(\alpha)\alpha f(\alpha) - c_{\alpha}c_{\ba}A_{\alpha}'f(\alpha)
= f(\alpha)c_{\ba}\ba - c_{\alpha}c_{\ba}A_{\gamma} =0$$
since $f(\alpha)\ba = c_{\gamma}A_{\gamma}$ and $c_{\alpha}=c_{\gamma}$. 
Hence $\vf J$ is generated by $\vf \gamma$. 

\smallskip
The $g$-cycle of $\alpha$ is of the form $(\delta \ f(\alpha) \ \alpha \ldots)$

\smallskip
(i) \ Assume first that  this has length $3$, then $Q$ has only two vertices. In this case, by the assumption,
$m_{\alpha}n_{\alpha}\geq 4$ so that $m_{\alpha}\geq 2$. 
We have in this case 
%$A_{\alpha}'f(\alpha) = \mu \gamma^2 f(\alpha)$ 
$A_{\alpha}'\gamma = \mu \alpha \gamma^2$ 
for some monomial $\mu$, and 
%$\gamma^2f(\alpha)=0$ 
$\alpha \gamma^2 = 0$ 
(by relation (3) of Definition \ref{def:2.2}).
It follows that $\vf \La$ has basis $\{ \vf, f(\alpha)\alpha\gamma, \ldots, B_{f(\alpha)} \}$ and has the required dimension.

(ii) \  Now assume that $m_{\alpha}n_{\alpha} = 4$. Then $Q$ has three vertices and $g$ is equal to
$$(\ba) (\alpha \ \gamma \ \delta \ f(\alpha) (f(\gamma)).
$$
We consider first the case when $f(\gamma)$ also is virtual, that is $\La$ is a triangle algebra.
Then we have using the relations that
$$A_{\alpha}'\gamma = \gamma\delta\gamma = 
c_{f(\gamma)}c_{f(\alpha)} A_{f(\alpha)}
$$
and $A_{f(\alpha)} = f(\alpha) \alpha\gamma$. It follows that
$$\vf\gamma = (1- c_{\ba}c_{\alpha}^2c_{f(\gamma)})f(\alpha)\alpha\gamma.
$$
We assume that $\La$ is not the singular triangle algebra, which means that the scalar in this identity is non-zero.
It follows that $\vf\La$ has basis $\{ \vf, \ f(\alpha)\alpha\gamma, \ B_{f(\alpha)}\}$ of the required size.

\medskip

(iii) Now assume $m_{\alpha}n_{\alpha}\geq 5$ so that $A_{\alpha}'$ has length $\geq 3$. Then we have
$$A_{\alpha}'\gamma = \mu g^{-1}(\delta) \delta f(\delta)
$$ 
with $\mu$ a monomial of length $\geq 1$
which is either zero or $\equiv A_{\beta}$ for $\beta = g^{-1}(\delta)$ or $\overline{g^{-1}(\delta)}$. In the second case, $JA_{\beta}$ is in the
socle. Hence $\vf \gamma = f(\alpha)\alpha\gamma - \omega$ with $\omega$ in the socle or zero. 
Then $\vf\La$ has basis $\{ \vf, \vf\gamma,  \ f(\alpha)\alpha\gamma g(\gamma), \ldots, B_{f(\alpha)}\}$ and the dimension is
as stated.

Since $\vf f(\alpha) = 0$, we know $\Omega^3(S_i)$ contains $f(\alpha)\La$ which is isomorphic to $\Omega^{-1}(S_i)$. One sees that they have the
same dimension, hence they are equal.
\end{proof}

\bigskip

Now assume $\ba$ is virtual but not a loop. Then $\alpha$ is not virtual,
and it cannot be a loop (see 2.1.1 and 3.1).

Then the quiver around $\ba$ has the following form
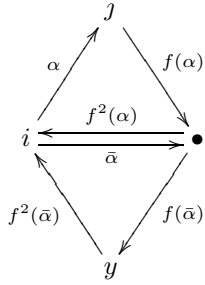
\[
%  \xymatrix@R=2pc@C=1.5pc{
%  \xymatrix@R=3.5pc@C=1.8pc{
  \xymatrix@R=3.pc@C=1.8pc{
%  \xymatrix@C=.8pc{
    & j
    \ar[rd]^{f(\alpha)}
    \\
    i
    \ar[ru]^{\alpha}
    \ar@<-.5ex>[rr]_{\ba}
    && \bullet
    \ar@<-.5ex>[ll]_{f^2(\alpha)}
 \ar[ld]^{f(\ba)}
    \\
    & y
    \ar[lu]^{f^2(\ba)}
  }
\]

\medskip

\begin{lemma}\label{lem:5.2}
Assume $\La$ is not the singular spherical algebra.
Then there is an exact sequence of $\La$-modules
	$$0\to \Omega^{-1}(S_i)\to P_y \to P_j \to \Omega(S_i)\to 0
	$$
	which gives rise to a minimal projective resolution of period $4$.
\end{lemma}

%\bigskip

\begin{proof} 
We may assume that $y\neq j$, otherwise $Q$ is the triangular quiver, and from Example 3.4 in \cite{WSA-GV} we know that the
algebra occurs already in Lemma \ref{lem:5.1}, part (ii).
We identify $\Omega(S_i) = \alpha \La$ and then 
$\Omega^2(S_i) = \{ x\in e_j\La\mid \alpha x = 0\}$. 
We have the following relations in $\La$:
\begin{enumerate}[(i)]
 \item
    $\alpha f(\alpha) = c_{\ba} \ba$,
 \item
    $\ba f(\ba) = c_{\alpha}A_{\alpha}$.
\end{enumerate}
Hence 
$c_{\alpha}A_{\alpha} = \ba f(\ba) 
 = c_{\ba}^{-1} \alpha f(\alpha)f(\ba)$ 
and if we set
$$\vf:= f(\alpha)f(\ba)   - c_{\ba}c_{\alpha} A_{\alpha}'$$
(where $\alpha A_{\alpha}' = A_{\alpha}$), 
then $\vf \La \subseteq \Omega^2(S_i)$.

The module $\Omega^2(S_i)$ has dimension 
$m_{f(\alpha)}n_{f(\alpha)}-1$. 
We want to show that $\vf \La$ has the same dimension.
%Assume first that  $j\neq y$. 
Let $\gamma = g(f(\bar{\alpha}))$ and $\delta = f^{-1}(\gamma)$.

\smallskip

\ As in \cite[Lemma~5.6]{WSA-GV} we have $\vf f^2(\ba)=0$.
Hence $\vf \rad \La$ is generated by $\vf \gamma$.

\smallskip

(a) \ Assume first 
that $m_{\alpha}n_{\alpha}=3$, then $n_{f(\alpha)} \geq 5$ (see Observation~\ref{obs:2.3}(1)). 
The permutation $g$ has a cycle $(\alpha \ g(\alpha) \ f^2(\ba))$
and a cycle $(f(\gamma) \ f(\alpha) \ f(\ba) \ \gamma \ \ldots)$. 

\medskip
We have   $A_{\alpha}' = g(\alpha)$ and 
$$g(\alpha)\gamma = g(\alpha)f(g(\alpha)) = c_{f(\alpha)}A_{f(\alpha)} = f(\alpha)f(\ba)\gamma\mu
$$
with $\mu$ a monomial of length $\geq 1$.
Therefore we can write 
$$\vf\gamma = f(\alpha)f(\ba)\gamma (1-\lambda \mu)
$$
and $1-\lambda\mu$ is a unit in $\La$. 
Moreover, we compute 
$$f(\alpha)f(\ba)\gamma f(\gamma) \equiv  f(\alpha)f(\ba)A_{f^2(\ba)} \ 
\equiv \ f(\alpha)f^2(\alpha)\alpha \equiv B_{f(\alpha)}.
$$
From these it follows that $\vf\La$ has basis
$\{ \vf, f(\alpha)f(\ba) \gamma, \ldots, B_{f(\alpha)}\}$ of size 
$m_{f(\alpha)}n_{f(\alpha)} -1$, as required.

\bigskip

(b)  Assume $m_{\alpha}n_{\alpha} = 4 = n_{\alpha}$ so that 
$\vf = f(\alpha)f(\ba) - c_{\ba}c_{\alpha} g(\alpha)g^2(\alpha)$. 

Consider first the case where in addition $f(g(\alpha))$ is virtual, then it cannot be a loop
(otherwise $Q$ would not be 2-regular). 
We find that the quiver is then the spherical quiver, which has
two $g$-cycles of length four and two pairs of virtual arrows. 

Let $\gamma = f(g^2(\alpha))$. Then one finds  
$\vf\gamma  = f(\alpha)f(\ba)\gamma - \lambda g(\alpha)f(g(\alpha)$ for a non-zero scalar. As before this can be written as
$$\vf\gamma = f(\alpha)f(\ba)\gamma(1-a\mu)
$$
with $(1-a\mu)$ a unit (we exclude the singular spherical algebra).  Then  one gets a basis for $\vf\La$ as in part (a), of
the right size.

If $f(g(\alpha))$ is not virtual then $A_{\alpha}'\gamma = g(\alpha)g^2(\alpha) f(g^2(\alpha))$ which is zero, and again
$\vf\La$ has the required dimension.

(c) Now assume that $m_{\alpha}n_{\alpha}\geq 5$, then $A_{\alpha}'\gamma = \mu g^{-1}(\beta) \beta f(\beta)$ where $\beta$ is
the last arrow in $A_{\alpha}'$, and where $\mu$ has length $\geq 2$. Now either $g^{-1}(\beta)\beta f(\beta)=0$, or it is a non-zero scalar multiple
of some 
$A_{\delta}$ by Lemma \ref{lem:3.2},  and then $J^2A_{\delta}=0$ by Lemma \ref{lem:4.3}. It follows as before that $\vf\La$ has dimension $m_{f(\alpha)}n_{f(\alpha)}-1$.
\bigskip

Since $\vf f^2(\ba)=0$ we know $\Omega^3(S_i)$ contains $f^2(\ba)\La$ 
and this is isomorphic to $\Omega^{-1}(S_i)$. 
One sees that they have the same dimension and hence
are equal. 
\end{proof}

\bigskip

\subsection{Section 6}

The only changes needed are as follows.

(1) \ In Lemma 6.4, the zero relations must be
$$\theta f(\theta)g(f(\theta)) = 0 \ \mbox{for} \ \theta \in Q_1\setminus\{ \beta, \delta, \omega\} \ \mbox{ and } \ 
\theta g(\theta)f(g(\theta)) = 0 \ \mbox{for} \ \theta \in Q_1\setminus\{ \alpha, \sigma, \delta\}.
$$

(2) \ In Lemma 6.5, the zero relations must be
$$\theta f(\theta)g(f(\theta)) = 0 \ \mbox{for} \ \theta \in Q_1\setminus\{ \beta, \delta, \mu, \omega\} \ \mbox{ and } \ 
\theta g(\theta)f(g(\theta)) = 0 \ \mbox{for} \ \theta \in Q_1\setminus\{ \alpha, \sigma, \varepsilon,  \delta\},
$$
and also  $\mu\varepsilon\mu=0$ and $\varepsilon\mu\varepsilon=0$ if $r\geq 3$.

\end{document}